\documentclass{scrartcl}

\setkomafont{disposition}{\normalfont\bfseries}

\usepackage{latexsym,amsmath,amssymb,amsthm,amsfonts,bm}
\usepackage[active]{srcltx}     %DVI-inverse search
\usepackage[all]{xy}
\usepackage{verbatim}
\usepackage{pict2e}
\usepackage{graphicx}
\usepackage{float}

\renewcommand{\epsilon}{\varepsilon}
\renewcommand{\Im}{\mathrm{Im}\,}

\newcommand{\R}{\mathbb{R}}

\newcommand{\N}{\mathbb{N}}
\newcommand{\Z}{\mathbb{Z}}
\newcommand{\C}{\mathbb{C}}
\newcommand{\h}{\mathbb{H}}

\newtheorem{theorem}{Theorem}

\newtheorem{proposition}{Proposition}

\author{Maryna S. Viazovska }
\title{The sphere packing problem in dimension 8}
\begin{document}

\maketitle

\begin{abstract}
In this paper we prove that no packing of unit balls
in Euclidean space $\R^8$ has density greater than that of the $E_8$-lattice packing.
\end{abstract}


\section{Introduction}
The sphere packing constant measures which portion of $d$-dimensional Euclidean
space can be covered by non-overlapping unit balls. More precisely, let $\R^d$ be the Euclidean vector space equipped with distance $\|\cdot\|$ and Lebesgue measure $\mathrm{Vol}(\cdot)$. For $x\in\R^d$ and $r\in\R_{>0}$ we denote by $B_d(x,r)$ the open ball in $\R^d$ with center $x$ and radius $r$.
Let $X\subset \R^d$ be a discrete set of points such that $\|x-y\|\geq2$ for any distinct $x,y\in X$. Then the union
$$\mathcal{P}=\bigcup_{x\in X} B_d(x,1)$$ is a \emph{sphere packing}. If $X$ is a lattice in $\R^d$ then we say that $\mathcal{P}$ is a \emph{lattice sphere packing}. The \emph{finite density} of a packing $\mathcal{P}$ is defined as
$$\Delta_{\mathcal{P}}(r):=\frac{\mathrm{Vol}(\mathcal{P}\cap B_d(0,r))}{\mathrm{Vol}(B_d(0,r))},\quad r>0.$$
We define the \emph{density} of a packing $\mathcal{P}$ as the limit superior
$$\Delta_{\mathcal{P}}:=\limsup\limits_{r\to\infty}\Delta_{\mathcal{P}}(r). $$
The number be want to know is the supremum over all possible packing densities
$$\Delta_d:=\sup\limits_{\substack{\mathcal{P}\subset\R^d\\ \scriptscriptstyle\mathrm{sphere}\;\mathrm{packing}}}\Delta_{\mathcal{P}}, $$
                         called the \emph{sphere packing constant}.

For which dimensions do we know the exact value of $\Delta_d$? Trivially, in dimension~$1$ we have $\Delta_1=1$. It has long been known  that a best packing in dimension $2$ is the familiar hexagonal lattice packing, in which each disk is touching six others. The first proof of this result was given by A. Thue at the beginning ot twentieth century \cite{Thue}. However, his proof was considered by some experts incomplete. A rigorous proof was given by L. Fejes T\'{o}th in 1940s \cite{Toth}. The density of the hexagonal lattice packing is~$\frac{\pi}{\sqrt{12}}$, therefore  $\Delta_2=\frac{\pi}{\sqrt{12}}\approx0.90690$. The packing problem in dimension~3 turned out to be more difficult. Johannes Kepler conjectured in his essay ``On the six-cornered snowflake'' (1611) that no arrangement of equally sized spheres filling space has density greater than~$\frac{\pi}{\sqrt{18}}$.  This density is attained by the face-centered cubic packing and also by uncountably many non-lattice packings. The Kepler conjecture was famously proven by T. Hales in 1998 \cite{Hales} and therefore we know that $\Delta_3=\frac{\pi}{\sqrt{18}}\approx0.74048$. In
2015 Hales and his 21 coauthors published a complete formal proof of the Kepler conjecture that can be  verified by automated proof checking software. Before now, the exact values of the sphere packing constants in all dimensions greater than~3 have been unknown. A list of conjectural best packings in dimensions less than~10 can be found in~\cite{ConwaySloan}. Upper bounds for the sphere packing constants $\Delta_d$ as $d\leq 36$ are given in~\cite{ElkiesCohn}. Surprisingly enough, these upper bounds and known lower bounds on $\Delta_d$ are extremely close in dimensions~$d=8$ and~$d=24$.

The main result of this paper is the proof that $$\Delta_8=\frac{\pi^4}{384}\approx 0.25367.$$
This is the density of the $E_8$-lattice sphere packing. Recall that the $E_8$-lattice $\Lambda_8\subset\R^8$ is given by
$$\Lambda_8=\{(x_i)\in\Z^8\cup(\Z+\textstyle\frac12\displaystyle )^8|\;\sum_{i=1}^8x_i\equiv 0\;(\mathrm{mod\;2})\}.$$
$\Lambda_8$ is the unique up to isometry positive-definite, even, unimodular lattice of rank 8. The name derives from the fact that it is the root lattice of the $E_8$ root system. The minimal distance between two points in $\Lambda_8$ is $\sqrt{2}$.
The $E_8$-lattice sphere packing is the packing of unit balls with centers at $\frac{1}{\sqrt{2}}\Lambda_8.$ Our main result is
\begin{theorem}\label{thm: main}
 No packing of unit balls
in Euclidean space $\R^8$ has density greater than that of the $E_8$-lattice packing.
\end{theorem}
Furthermore, our proof of Theorem~1 combined with arguments given in \cite[Section~8]{ElkiesCohn} implies that the $E_8$-lattice sphere packing is the unique periodic packing of maximal density.

The paper is organized as follows. In Section~2 we explain the idea of the proof of Theorem~\ref{thm: main} and describe the methods we use. In Section~3 we give a brief overview of the theory of modular forms. In Section~4 we construct supplementary radial functions $a$, $b:\R^8\to i\R$, which are eigenfunctions of the Fourier transform and have double zeroes at almost all points of $\Lambda_8$. This construction is crucial for our proof of Theorem~\ref{thm: main}.  Finally, in Section~5 we complete the proof.

\section{Linear programming bounds}
 Our proof of Theorem~\ref{thm: main} is based on linear programming bounds. This technique was successfully applied to obtain upper bounds in a wide range of discrete optimization problems such as error-correcting codes \cite{Delsarte72}, equal weight quadrature formulas \cite{Delsarte77}, and spherical codes \cite{KabLev, PfenderZiegler}. In  exceptional cases linear programming bounds are optimal \cite{CohnKumar}. However, in general linear programming bounds are not sharp and it is an open question how big the errors of such bounds can be. It is known \cite{BRV} that the linear programming bounds for the minimal number of points in an equal weight quadrature formula on the sphere $S^d$ are asymptotically optimal up to a constant depending on $d$. Linear programming bounds can also be applied to the sphere packing problem. Kabatiansky and Levenshtein \cite{KabLev} deduced upper bounds for sphere packing from their results on spherical codes.

 In 2003 Cohn and Elkies \cite{ElkiesCohn}  developed  linear programming bounds that apply directly to sphere packings. Using their new method they improved the previously known upper bounds for the sphere packing constant in dimensions from $4$ to $36$. The most striking results obtained by this technique are upper bounds for dimensions $8$ and $24$. For example, their upper bound for $\Delta_8$ was only
$1.000001$ times greater than the lower bound, which is given by the density of the $E_8$ sphere packing. This bound can be improved even further by more extensive computer computations. %\cite{Cohn}.

We explain the Cohn--Elkies linear programming bounds in more detail. To this end we recall a few definitions from Fourier analysis. The \emph{Fourier transform} of an $L^1$~function $f:\R^d\to\C$ is defined as
$$\mathcal{F}(f)(y)=\widehat{f}(y):=\int\limits_{\R^d} f(x)\,e^{-2\pi i x\cdot y}\,dx,\quad y\in\R^d $$
where $x\cdot y=\frac12\|x\|^2+\frac12\|y\|^2-\frac12\|x-y\|^2$ is the standard scalar product in $\R^d$.
A $C^\infty$~function $f:\R^d\to\C$ is called a \emph{Schwartz function} if it tends to zero as $\|x\|\to\infty$ faster then any inverse power of $\|x\|$, and the same holds for all partial derivatives of $f$. The set of all Schwartz functions is called the \emph{Schwartz space}. The Fourier transform is an automorphism of this space. We will also need the following wider class of functions. We say that a function $f:\R^d\to\C$ is \emph{admissible} if there is a constant $\delta>0$ such that $|f(x)|$ and $|\widehat{f}(x)|$ are bounded above by a constant times $(1+|x|)^{-d-\delta}$.
The following theorem is the key result of \cite{ElkiesCohn}:
\begin{theorem}\label{thm: Cohn-Elkies} (Cohn, Elkies \cite{ElkiesCohn}) Suppose that  $f:\R^d\to\R$ is an admissible function, is not identically zero, and satisfies:
\begin{equation}\label{eqn: Cohn-Elkies condition 1}f(x)\leq 0\mbox{ for } \|x\|\geq 1\end{equation} and
\begin{equation}\label{eqn: Cohn-Elkies condition 2}\widehat{f}(x)\geq0\mbox{ for all } x\in\R^d.\end{equation}
 Then the  density of  $d$-dimensional
 sphere packings is bounded above by $$\frac{f(0)}{\widehat{f}(0)}\cdot \frac{\pi^{\frac{d}{2}}}{2^d\,\Gamma(\textstyle \frac{d}{2}+1)}=\frac{f(0)}{\widehat{f}(0)}\cdot \mathrm{Vol}\,B_d(0,\frac 12).$$
\end{theorem}

 Without loss of generality we can assume that a function $f$ in Theorem~\ref{thm: Cohn-Elkies} is radial, i.~e. its value at each point depends only on the distance between the point and the origin \cite[p. 695]{ElkiesCohn}. For a radial function $f_0:\R^d\to\R$ we will denote by $f_0(r)$ the common value of $f_0$ on vectors of length $r$. Henceforth we assume $d=8$. The Poisson summation formula implies
 $$ \sum_{\ell\in\frac{1}{\sqrt{2}}\Lambda_8}f(\ell)= 2^4\,\sum_{\ell\in\sqrt{2}\Lambda_8}\widehat{f}(\ell).$$
 Hence, if a function $f$ satisfies conditions \eqref{eqn: Cohn-Elkies condition 1} and \eqref{eqn: Cohn-Elkies condition 2} then
 $$ \frac{f(0)}{\widehat{f}(0)}\geq 2^4.$$
 We say that an admissible function $f:\R^8\to\R$ is \emph{optimal} if it satisfies \eqref{eqn: Cohn-Elkies condition 1}, \eqref{eqn: Cohn-Elkies condition 2} and $f(0)/\widehat{f}(0)=2^4$.

 The main step in our proof of Theorem \ref{thm: main} is the explicit  construction of an optimal function. It will be convenient for us to scale this function by $\sqrt{2}$.
\begin{theorem}\label{thm: g}
There exists a radial Schwartz function $g:\R^8\to\R$ which satisfies:
\begin{align}
g(x)&\leq 0\mbox{ for } \|x\|\geq \sqrt{2}, \label{eqn: g1}\\
\widehat{g}(x)&\geq0\mbox{ for all } x\in\R^8,\label{eqn: g2}\\
g(0)&=\widehat{g}(0)=1.\label{eqn: g3}
\end{align}
Moreover, the values $g(x)$ and $\widehat{g}(x)$ do not vanish for all vectors $x$ with $\|x\|^2\notin 2\Z_{>0}$.
\end{theorem}
Theorem \ref{thm: Cohn-Elkies} applied to the optimal function $f(x)=g(\sqrt{2}\,x)$ immediately implies Theorem \ref{thm: main}. Additionally, the function $g$ satisfies the conclusions of \cite[Conjecture~8.1]{ElkiesCohn}. This implies the uniqueness of the densest periodic sphere packing in $\R^8$.

Let us briefly explain our strategy for the proof of Theorem~\ref{thm: g}. First, we observe that conditions \eqref{eqn: g1}--\eqref{eqn: g3} imply additional properties of the function $g$. Suppose that there exists a Schwartz function $g$ such that the conditions \eqref{eqn: g1}--\eqref{eqn: g3} hold. The Poisson summation formula states
\begin{equation}
 \sum_{\ell\in\Lambda_8}g(\ell)= \sum_{\ell\in\Lambda_8}\widehat{g}(\ell).
\end{equation}
Since $\|\ell\|\geq\sqrt{2}$ for all $\ell\in\Lambda_8\backslash\{0\}$, conditions \eqref{eqn: g1} and \eqref{eqn: g3} imply
\begin{equation}
  \sum_{\ell\in\Lambda_8}g(\ell)\leq g(0)=1.
\end{equation}
On the other hand, conditions \eqref{eqn: g2} and \eqref{eqn: g3} imply
\begin{equation}
  \sum_{\ell\in\Lambda_8}\widehat{g}(\ell)\geq \widehat{g}(0)=1.
\end{equation}
Therefore, we deduce that $g(\ell)=\widehat{g}(\ell)=0$ for all $\ell\in\Lambda_8\backslash\{0\}$. Moreover, the first derivatives $\frac{d}{dr}g(r)$ and $\frac{d}{dr}\widehat{g}(r)$ also vanish at all $\Lambda_8$-lattice points of length bigger than $\sqrt{2}$. We will say that $g$ and $\widehat{g}$ have double zeroes at these points. This property gives us a hint on constructing the function $g$ explicitly.

In Section \ref{sec: g} a function $g$ satisfying \eqref{eqn: g1}--\eqref{eqn: g3} is given in a closed form. Namely, it is defined as an integral transform (Laplace transform) of a \emph{modular form} of a certain kind. The next section is a brief introduction to the theory of modular forms.

\section{Modular forms}
Let $\h$ be the upper half-plane $\{z\in\C\mid\Im(z)>0\}$. The modular group $\Gamma(1):=\mathrm{PSL}_2(\Z)$ acts on $\h$ by linear fractional transformations
$$\left(\begin{smallmatrix}a&b\\c&d\end{smallmatrix}\right)z:=\frac{az+b}{cz+d}.$$

Let $N$ be a positive integer. The \emph{level $N$ principal congruence subgroup} of $\Gamma(1)$ is
$$\Gamma(N):=\left\{\left.\left(\begin{smallmatrix}a&b\\c&d\end{smallmatrix}\right)\in\Gamma(1)\right|\left(\begin{smallmatrix}a&b\\c&d\end{smallmatrix}\right)\equiv\left(\begin{smallmatrix}1&0\\0&1\end{smallmatrix}\right)\;\mathrm{mod}\;N\right\}.$$
A subgroup $\Gamma\subset\Gamma(1)$ is called a \emph{congruence subgroup} if $\Gamma(N)\subset\Gamma$ for some $N\in\N$. An important example of a congruence subgroup is
$$\Gamma_0(N):=\left\{\left.\left(\begin{smallmatrix}a&b\\c&d\end{smallmatrix}\right)\in\Gamma(1)\right|\;c\equiv0\;\mathrm{mod}\;N\right\}.$$

Let $z\in\h$, $k\in\Z$, and $\left(\begin{smallmatrix}a&b\\c&d\end{smallmatrix}\right)\in\mathrm{SL}_2(\Z)$. The \emph{automorphy factor} of weight $k$ is defined as
$$j_k(z,\left(\begin{smallmatrix}a&b\\c&d\end{smallmatrix}\right)):=(cz+d)^{-k}.$$
The automorphy factor satisfies the \emph{chain rule}
$$j_k(z,\gamma_1\gamma_2)=j_k(z,\gamma_2)\,j_k(\gamma_2z,\gamma_1). $$
Let $F$ be a  function on $\h$ and $\gamma\in\mathrm{PSL}_2(\Z)$. Then the \emph{slash operator} acts on $F$ by
$$(F|_k\gamma)(z):=j_k(z,\gamma)\,F(\gamma z). $$
The chain rule implies
$$F|_k\gamma_1\gamma_2=(F|_k\gamma_1)|_k\gamma_2.$$

A \emph{(holomorphic) modular form} of integer weight $k$ and congruence subgroup $\Gamma$ is a holomorphic function $f:\h\to\C$ such that:
\begin{enumerate}
 \item $f|_k\gamma=f$ for all $\gamma\in\Gamma$ and
 \item for each $\alpha\in\Gamma(1)$ the function $f|_k\alpha$ has Fourier expansion $$f|_k\alpha (z)=\sum_{n=0}^\infty c_f(\alpha,\frac{n}{n_\alpha})\,e^{2\pi i \frac{n}{n_\alpha}z}$$ for some $n_\alpha\in\N$ and Fourier coefficients $c_f(\alpha,m)\in\C$.
\end{enumerate}

Let $M_k(\Gamma)$ be the space of modular forms of weight $k$ for the congruence subgroup $\Gamma$. A key fact in the theory of modular forms is that the spaces $M_k(\Gamma)$ are finite dimensional.

We consider several examples of modular forms. For an even integer $k\geq 4$ we define the \emph{weight $k$ Eisenstein series} as
\begin{equation}\label{eqn: Ek definition}E_k(z):=\frac{1}{2\zeta(k)}\sum_{(c,d)\in\Z^2\backslash(0,0)}(cz+d)^{-k}.\end{equation}
Since the sum converges absolutely, it is easy to see that $E_k\in M_k(\Gamma(1))$. The
Eisenstein series possesses the Fourier expansion
\begin{equation}\label{eqn: Ek Fourier}E_k(z)=1+\frac{2}{\zeta(1-k)}\sum_{n=1}^\infty \sigma_{k-1}(n)\,e^{2\pi i n z}, \end{equation}
where $\sigma_{k-1}(n)\,=\,\sum_{d|n} d^{k-1}$. In particular, we have
\begin{align*}
  E_4(z)\,=\,& 1+240\sum_{n=1}^\infty \sigma_3(n)\,e^{2\pi i n z} ,\\
  E_6(z)\,=\,& 1-504\sum_{n=1}^\infty \sigma_5(n)\,e^{2\pi i n z}.
\end{align*}
The infinite sum \eqref{eqn: Ek definition} does not converge absolutely for $k=2$. On the other hand, the expression \eqref{eqn: Ek Fourier} converges to a holomorphic function on the upper half-plane and therefore we set
  \begin{equation}\label{eqn: E2 def}E_2(z):= 1-24\sum_{n=1}^\infty \sigma_1(n)\,e^{2\pi i n z}.\end{equation}
This function is not modular, but it satisfies
\begin{equation}\label{eqn: E2 transform}z^{-2}\,E_2\Big(\frac{-1}{z}\Big)=E_2(z) -\frac{6i}{\pi}\, \frac{1}{z}.\end{equation}
The proof of this identity can be found in \cite[Section~2.3]{1-2-3}.
The weight two Eisenstein series $E_2$ is an example of a \emph{quasimodular form} \cite[Section~5.1]{1-2-3}.

Another example of modular forms we consider are \emph{theta functions} \cite[Section~3.1]{1-2-3}.
We define three theta functions (so-called ``Thetanullwerte'') as
\begin{align*}
  \theta_{00}(z)\,=\, & \sum_{n\in\Z}e^{\pi i n^2 z}, \\
  \theta_{01}(z)\,=\, & \sum_{n\in\Z}(-1)^n\,e^{\pi i n^2 z}, \\
  \theta_{10}(z)\,=\, & \sum_{n\in\Z}e^{\pi i (n+\frac12)^2 z}.
\end{align*}
The group $\Gamma(1)$ is generated by the elements $T=\left(\begin{smallmatrix}1&1\\0&1\end{smallmatrix}\right)$ and $S=\left(\begin{smallmatrix}0&1\\-1&0\end{smallmatrix}\right)$. These elements  act on the fourth powers of the theta functions in the following way
\begin{align}
z^{-2}\,\theta^4_{00}\Big(\frac{-1}{z}\Big)\,=\,&-\theta_{00}^4(z), \label{eqn: theta transform S}\\
z^{-2}\,\theta^4_{01}\Big(\frac{-1}{z}\Big)\,=\,&-\theta_{10}^4(z),\\
z^{-2}\,\theta^4_{10}\Big(\frac{-1}{z}\Big)\,=\,&-\theta_{01}^4(z),
\end{align}
and
\begin{align}
\theta^4_{00}(z+1)\,=\,&\theta_{01}^4(z),\\
\theta^4_{01}(z+1)\,=\,&\theta_{00}^4(z),\\
\theta^4_{10}(z+1)\,=\,&-\theta_{10}^4(z). \label{eqn: theta transform T}
\end{align}
Moreover, these three theta functions satisfy the \emph{Jacobi identity}
\begin{equation}
\theta_{01}^4+\theta_{10}^4=\theta_{00}^4.
\end{equation}
The theta functions $\theta^4_{00},\theta^4_{01}$, and $\theta^4_{10}$ belong to $M_2(\Gamma(2))$.

A \emph{weakly-holomorphic modular form} of integer weight $k$ and congruence subgroup $\Gamma$ is a holomorphic function $f:\h\to\C$ such that:
\begin{enumerate}
 \item $f|_k\gamma=f$ for all $\gamma\in\Gamma$,
 \item for each $\alpha\in\Gamma(1)$ the function $f|_k\alpha$ has  Fourier expansion $$f|_k\alpha (z)=\sum_{n=n_0}^\infty c_f(\alpha,\frac{n}{n_\alpha})\,e^{2\pi i \frac{n}{n_\alpha}z}$$ for some $n_0\in\Z$ and $n_\alpha\in\N$.
\end{enumerate}
For an $m$-periodic holomorphic function $f$ and $n\in\frac 1m \Z$ we will denote the $n$-th Fourier coefficient of $f$ by $c_f(n)$ so that
$$f(z)=\sum_{n\in\frac 1m \Z} c_f(n)\,e^{2\pi i n z}.$$
We denote the space of weakly-holomorphic modular forms of weight $k$ and group $\Gamma$ by $M_k^!(\Gamma)$. The spaces $M_k^!(\Gamma)$ are infinite dimensional.
Probably the most famous weakly-holomorphic modular form is the \emph{elliptic j-invariant}
$$j\,:=\,\frac{1728\, E_4^3}{E_4^3-E_6^2}. $$
This function belongs to $M_0^!(\Gamma(1))$ and has the Fourier expansion
$$j(z)\,=\,q^{-1} + 744 + 196884\, q + 21493760\, q^2 + 864299970\, q^3 +
 20245856256\, q^4 + O(q^5) $$
 where $q=e^{2\pi i z}$. Using a simple computer algebra system such as PARI GP or Mathematica one can compute the first hundred terms of this Fourier expansion within a few seconds. An important question is to find an  asymptotic formula for $c_j(n)$, the $n$-th Fourier coefficient  of $j$. Using the Hardy-Ramanujan circle method \cite[p. 460 -- 461]{Rademacher38} or the non-holomorphic Poincar{\'e} series \cite{Petersson32} one can show that \begin{equation} \label{eqn: j Fourier asymptotic}
 c_j(n)=\frac{2\pi}{\sqrt{n}}\sum_{k=1}^\infty \frac{A_k(n)}{k}\,I_1\left(\frac{4\pi \sqrt{n}}{k}\right)\qquad n\in\Z_{>0}
\end{equation}
 where
 $$A_k(n)= \sum_{\substack{h\;\mathrm{mod}\;k\\(h,k)=1}} e^{\frac{-2\pi i}{k}(nh+h')},\quad hh'\equiv -1(\mbox{mod}\;k),$$
 and $I_\alpha(x)$ denotes the modified Bessel function of the first kind defined as in \cite[Section~9.6]{Abramowitz}. A similar convergent asymptotic expansion holds for the Fourier coefficients of any weakly holomorphic modular form \cite[p.660 -- 662]{Hejhal}, \cite[Propositions~1.10 and~1.12]{Bruinier}. Such a convergent expansion implies effective estimates for the Fourier coefficients.

 For a  comprehensive introduction to the theory of modular forms we refer the  reader to \cite{1-2-3} and \cite{first course}.

\section{Fourier eigenfunctions with double zeroes at lattice points}\label{sec: fourier double zeroes}
In this section we construct two radial Schwartz functions $a,b:\R^8\to i\R$ such that
\begin{align}\mathcal{F}(a)&=a\label{eqn: a fourier}\\
 \mathcal{F}(b)&=-b\label{eqn: b fourier}
\end{align}
which double zeroes at all $\Lambda_8$-vectors of length greater than $\sqrt{2}$. Recall that each vector of $\Lambda_8$ has length $\sqrt{2n}$ for some $n\in\N_{\geq 0}$. We define $a$ and $b$ so that their values are purely imaginary because this simplifies some of our computations. We will show in Section \ref{sec: g} that an appropriate linear combination of functions $a$ and $b$ satisfies conditions \eqref{eqn: g1}--\eqref{eqn: g3}.

First, we will define the function $a$. To this end we consider the following weakly holomorphic modular forms:
\begin{align} \varphi_{-2}\,:=\,&\frac{-1728 \,E_4\,E_6}{E_4^3-E_6^2},\\
 \varphi_{-4}\,:=\,&\frac{1728 \,E_4^2}{E_4^3-E_6^2}.
 \end{align}
  The modular form $E_4^3-E_6^2$ does not vanish in the upper half-plain, hence $\varphi_{-2}$ and $\varphi_{-4}$ have no poles in $\h$. Analogously to \eqref{eqn: j Fourier asymptotic}, the Fourier coefficients of $\varphi_{-2}$ and $\varphi_{-4}$ satisfy
  \begin{equation}\label{eqn: varphi Fourier asymptotic}
  c_{\varphi_\kappa}(n)=2\pi\,n^{\frac{\kappa-1}{2}}\sum_{k=1}^\infty \frac{A_k(n)}{k}\,I_{1-\kappa}\left(\frac{4\pi \sqrt{n}}{k}\right)\qquad n\in\Z_{>0},\;\kappa=-2,-4.
  \end{equation}
  We define
\begin{align}
  \phi_{-4}\,:= \,&\varphi_{-4},\label{eqn: def phi4}\\
  \phi_{-2}\,:= \,&\varphi_{-4}\,E_2+\varphi_{-2},\label{eqn: def phi2}\\
  \phi_{0}\,:= \,&\varphi_{-4}\,E_2^2+2\varphi_{-2}\,E_2+j-1728.\label{eqn: def phi0}
\end{align}
 The function $\phi_0(z)$ is not modular; however the identity \eqref{eqn: E2 transform} implies the following transformation rule:
\begin{equation}\label{eqn: phi0 transform}
\phi_0\Big(\frac{-1}{z}\Big)=\phi_0(z)-\frac{12i}{\pi}\,\frac{1}{z}\,\phi_{-2}(z)-\frac{36}{\pi^2}\,\frac{1}{z^2}\,\phi_{-4}(z).
\end{equation}
Moreover, we have
\begin{align}
\phi_{-2}\,=\,&-3\,D(\varphi_{-4})+3\varphi_{-2},\\
\phi_{0}\,=\,&12\,D^2(\varphi_{-4})-36\,D(\varphi_{-2})+24j-17856,\label{eqn: D varphi}
\end{align}
where $Df(z)=\frac{1}{2\pi i} \frac{d}{dz} f(z)$. These identities combined with \eqref{eqn: j Fourier asymptotic} and \eqref{eqn: varphi Fourier asymptotic} give the asymptotic formula for the Fourier coefficients $c_{\phi_{-4}}(n)$, $c_{\phi_{-2}}(n)$, and $c_{\phi_{0}}(n)$.
The  first several terms of the corresponding Fourier expansions are
\begin{align}
  \phi_{-4}(z)\,=\,&q^{-1} + 504 + 73764\, q + 2695040\, q^2 + 54755730\, q^3 +O(q^4),\label{eqn: phi fourier4}\\
  \phi_{-2}(z)\,=\,&720 + 203040\, q + 9417600\, q^2 + 223473600\, q^3 + 3566782080\, q^4+O(q^5),\label{eqn: phi fourier2}\\
  \phi_{0}(z)\,=\,&518400\, q + 31104000\, q^2 + 870912000\, q^3 + 15697152000\, q^4+O(q^5),\label{eqn: phi fourier0}
\end{align}
where $q=e^{2\pi i z}$.
For $x\in\R^8$ we define
\begin{align}\label{eqn: a(r) definition}
  a(x):=&\int\limits_{-1}^i\phi_0\Big(\frac{-1}{z+1}\Big)\,(z+1)^2\,e^{\pi i \|x\|^2 z}\,dz
  +\int\limits_{1}^i\phi_0\Big(\frac{-1}{z-1}\Big)\,(z-1)^2\,e^{\pi i \|x\|^2 z}\,dz\\
  -&2\int\limits_{0}^i\phi_0\Big(\frac{-1}{z}\Big)\,z^2\,e^{\pi i \|x\|^2 z}\,dz
  +2\int\limits_{i}^{i\infty}\phi_0(z)\,e^{\pi i \|x\|^2 z}\,dz.\nonumber
\end{align}
We observe that the contour integrals in \eqref{eqn: a(r) definition} converge absolutely and uniformly for  $x\in\R^8$. Indeed,
$\phi_0(z)=O(e^{-2\pi i z})$ as $\Im(z)\to \infty$. Therefore, $a(x)$ is well defined. Now we prove that $a$ satisfies condition \eqref{eqn: a fourier}.
\begin{proposition}\label{prop: a(r) Fourier}
The function $a$ defined by \eqref{eqn: a(r) definition} belongs to the Schwartz space and satisfies $$\widehat{a}(x)=a(x). $$
\end{proposition}
\begin{proof}
First, we prove that $a$ is a Schwartz function. From \eqref{eqn: j Fourier asymptotic}, \eqref{eqn: varphi Fourier asymptotic}, and \eqref{eqn: D varphi} we deduce that the Fourier coefficients of $\phi_0$ satisfy
$$|c_{\phi_0}(n)|\leq2\,e^{4\pi\sqrt{n}}\quad n\in\Z_{>0}.$$ Thus, there exists a positive constant $C$ such that
$$|\phi_0(z)|\leq C\,e^{-2\pi \Im{z}}\qquad \mbox{for } \; \Im{z}>\frac 12.$$ We estimate the first summand in the right-hand side of \eqref{eqn: a(r) definition}.  For $r\in\R_{\geq 0}$ we have
\begin{align*}&\Bigg|\int\limits_{-1}^{i}\phi_0\Big(\frac{-1}{z+1}\Big)\,(z+1)^2\,e^{\pi i r^2 z}\,dz\Bigg|=\Bigg|\int\limits_{i\infty}^{-1/(i+1)}\phi_0(z)\,z^{-4}\,e^{\pi i r^2 (-1/z-1)}\,dz\Bigg|\leq \\
 &C_1\int\limits_{1/2}^{\infty}e^{-2\pi t}\,e^{-\pi  r^2/t}\,dt\leq C_1\int\limits_{0}^{\infty}e^{-2\pi t}\,e^{-\pi  r^2/t}\,dt=C_2\,r\,K_1(2\sqrt{2}\,\pi\,r)
\end{align*}
where $C_1$ and $C_2$ are some positive constants and $K_\alpha(x)$ is the modified Bessel function of the second kind defined as in \cite[Section~9.6]{Abramowitz}. This estimate also holds for the second and third summand in \eqref{eqn: a(r) definition}.
For the last summand we have
$$ \Bigg|\int\limits_{i}^{i\infty}\phi_0(z)\,e^{\pi i r^2 z}\,dz\Bigg|\leq C\,\int\limits_{1}^{\infty} e^{-2\pi t}\,e^{-\pi r^2 t}\,dt=C_3\frac{e^{\pi(r^2+2)}}{r^2+2}.$$
Therefore, we arrive at
$$|a(r)|\leq 4C_2\,r\,K_1(2\sqrt{2}\pi r)+2C_3\frac{e^{-\pi(r^2+2)}}{r^2+2}.$$
It is easy to see that the left hand side of this inequality decays faster then any inverse power of $r$. Analogous estimates can be obtained for all derivatives $\frac{d^k}{dr^k}a(r)$.

Now we show that $a$ is an eigenfunction of the Fourier transform. We recall that the Fourier transform of a Gaussian function is
\begin{equation}\label{eqn: Gaussian Fourier}\mathcal{F}(e^{\pi i  \|x\|^2 z})(y)=z^{-4}\,e^{\pi i \|y\|^2 \,(\frac{-1}{z}) }.\end{equation}
Next, we exchange the contour integration with respect to $z$ variable and Fourier transform  with respect to $x$ variable in \eqref{eqn: a(r) definition}. This can be done, since the corresponding double integral converges absolutely. In this way we obtain
\begin{align*}
  \widehat{a}(y)=&\int\limits_{-1}^i\phi_0\Big(\frac{-1}{z+1}\Big)\,(z+1)^2\,z^{-4}\,e^{\pi i \|y\|^2 \,(\frac{-1}{z})}\,dz
  +\int\limits_{1}^i\phi_0\Big(\frac{-1}{z-1}\Big)\,(z-1)^2\,z^{-4}\,e^{\pi i \|y\|^2 \,(\frac{-1}{z})}\,dz\\
  -&2\int\limits_{0}^i\phi_0\Big(\frac{-1}{z}\Big)\,z^2\,z^{-4}\,e^{\pi i \|y\|^2 \,(\frac{-1}{z})}\,dz +2\int\limits_{i}^{i\infty}\phi_0(z)\,z^{-4}\,e^{\pi i \|y\|^2 \,(\frac{-1}{z})}\,dz.
\end{align*}
Now we make a change of variables $w=\frac{-1}{z}$. We obtain
\begin{align*}
  \widehat{a}(y)=&\int\limits_{1}^i\phi_0\Big(1-\frac{1}{w-1}\Big)\,(\frac{-1}{w}+1)^2\,w^{2}\,e^{\pi i \|y\|^2 \,w}\,dw\\
  +&\int\limits_{-1}^i\phi_0\Big(1-\frac{1}{w+1}\Big)\,(\frac{-1}{w}-1)^2\,w^2\,e^{\pi i \|y\|^2 \,w}\,dw\\
  -&2\int\limits_{i \infty}^i\phi_0(w)\,e^{\pi i \|y\|^2 \,w}\,dw +2\int\limits_{i}^{0}\phi_0\Big(\frac{-1}{w}\Big)\,w^{2}\,e^{\pi i \|y\|^2 \,w}\,dw.
\end{align*}
Since $\phi_0$ is $1$-periodic we have
\begin{align*}
  \widehat{a}(y)\,=\,&\int\limits_{1}^i\phi_0\Big(\frac{-1}{z-1}\Big)\,(z-1)^2\,e^{\pi i \|y\|^2 \,z}\,dz
  +\int\limits_{-1}^i\phi_0\Big(\frac{-1}{z+1}\Big)\,(z+1)^2\,e^{\pi i \|y\|^2 \,z}\,dz\\
  +&2\int\limits_{i}^{i\infty}\phi_0(z)\,e^{\pi i \|y\|^2 \,z}\,dz
  -2\int\limits_{0}^{i}\phi_0\Big(\frac{-1}{z}\Big)\,z^{2}\,e^{\pi i \|y\|^2 \,z}\,dz\\
  \,=\,&a(y).
\end{align*}
This finishes the proof of the proposition.
\end{proof}

Next, we check that $a$ has double zeroes at all $\Lambda_8$-lattice points of length greater then $\sqrt{2}$.
%Note that by definition function $a$ is radial and therefore in naturally defines a function on $\R_{\geq0}$. For abuse of notation we denote this function also by $a$.

\begin{proposition}\label{prop: a(r) double zeroes} For $r>\sqrt{2}$ we can express $a(r)$ in the following form
\begin{equation}\label{eqn: a double zeroes}
  a(r)=-4\sin(\pi r^2/2)^2\,\int\limits_{0}^{i\infty}\phi_0\Big(\frac{-1}{z}\Big)\,z^2\,e^{\pi i r^2 \,z}\,dz.
\end{equation}
\end{proposition}
\begin{proof}
We denote the right hand side of \eqref{eqn: a double zeroes} by $d(r)$.  It is easy to see that $d(r)$ is well-defined. Indeed, from the transformation formula \eqref{eqn: phi0 transform} and the expansions \eqref{eqn: phi fourier0}--\eqref{eqn: phi fourier4} we obtain
\begin{align*}
\phi_0\Big(\frac{-1}{it}\Big)=&O(e^{-2\pi/t})\quad\mbox{as}\;t\to 0\\
\phi_0\Big(\frac{-1}{it}\Big)=&O(t^{-2}\,e^{2\pi t})\quad\mbox{as}\;t\to \infty
\end{align*}
Hence, the integral \eqref{eqn: a double zeroes} converges absolutely for $r>\sqrt{2}$.
 We can write %\texttt{check signs}
\begin{align*}
 d(r)=&\int\limits_{-1}^{i\infty-1}\phi_0\Big(\frac{-1}{z+1}\Big)\,(z+1)^2\,e^{\pi i r^2 \,z}\,dz-
 2\int\limits_{0}^{i\infty}\phi_0\Big(\frac{-1}{z}\Big)\,z^2\,e^{\pi i r^2 \,z}\,dz\\
 +&\int\limits_{1}^{i\infty+1}\phi_0\Big(\frac{-1}{z-1}\Big)\,(z-1)^2\,e^{\pi i r^2 \,z}\,dz.
\end{align*}
From \eqref{eqn: phi0 transform} we deduce that if $r>\sqrt{2}$ then
$\phi_0\Big(\frac{-1}{z}\Big)\,z^2\,e^{\pi i r^2 \,z}\to 0$ as $\Im(z)\to\infty$. Therefore, we can deform the paths of integration
and rewrite
\begin{align*}
 d(r)=&\int\limits_{-1}^{i}\phi_0\Big(\frac{-1}{z+1}\Big)\,(z+1)^2\,e^{\pi i r^2 \,z}\,dz
 +\int\limits_{i}^{i\infty}\phi_0\Big(\frac{-1}{z+1}\Big)\,(z+1)^2\,e^{\pi i r^2 \,z}\,dz\\
 -2&\int\limits_{0}^{i}\phi_0\Big(\frac{-1}{z}\Big)\,z^2\,e^{\pi i r^2 \,z}\,dz
 -2\int\limits_{i}^{i\infty}\phi_0\Big(\frac{-1}{z}\Big)\,z^2\,e^{\pi i r^2 \,z}\,dz\\
 +&\int\limits_{1}^{i}\phi_0\Big(\frac{-1}{z-1}\Big)\,(z-1)^2\,e^{\pi i r^2 \,z}\,dz
 +\int\limits_{i}^{i\infty}\phi_0\Big(\frac{-1}{z-1}\Big)\,(z-1)^2\,e^{\pi i r^2 \,z}\,dz.
\end{align*}
Now from \eqref{eqn: phi0 transform} we find
\begin{align*}&\phi_0\Big(\frac{-1}{z+1}\Big)\,(z+1)^2-2\phi_0\Big(\frac{-1}{z}\Big)\,z^2+
\phi_0\Big(\frac{-1}{z-1}\Big)\,(z-1)^2=\\
&\phi_0(z+1)\,(z+1)^2-2\phi_0(z)\,z^2+\phi_0(z-1)\,(z-1)^2\\
&-\frac{12i}{\pi}\,\Big(\phi_{-2}(z+1)\,(z+1)-2\phi_{-2}(z)\,z+\phi_{-2}(z-1)\,(z-1)\Big)\\
&-\frac{36}{\pi^2}\Big(\phi_{-4}(z+1)-2\phi_{-4}(z)+\phi_{-4}(z-1)\Big)=\\
&2\phi_0(z).
 \end{align*}
 Thus, we obtain
 \begin{align*}
 d(r)=&\int\limits_{-1}^{i}\phi_0\Big(\frac{-1}{z+1}\Big)\,(z+1)^2\,e^{\pi i r^2 \,z}\,dz
 -2\int\limits_{0}^{i}\phi_0\Big(\frac{-1}{z}\Big)\,z^2\,e^{\pi i r^2 \,z}\,dz\\
 +&\int\limits_{1}^{i}\phi_0\Big(\frac{-1}{z-1}\Big)\,(z-1)^2\,e^{\pi i r^2 \,z}\,dz
 +2\int\limits_{i}^{i\infty}\phi_0(z)\,e^{\pi i r^2 \,z}\,dz=a(r).
\end{align*}
This finishes the proof.
\end{proof}
Finally, we find another convenient integral representation for $a$ and compute values of $a(r)$ at $r=0$ and $r=\sqrt{2}$.
\begin{proposition}\label{prop: a another integral}
For $r\geq0$ we have
\begin{align}\label{eqn: a another integral}a(r)=&4i\,\sin(\pi r^2/2)^2\,\Bigg(\frac{36}{\pi^3\,(r^2-2)}-\frac{8640}{\pi^3\,r^4}+\frac{18144}{\pi^3\,r^2}\\ +&\int\limits_0^\infty\,\left(t^2\,\phi_0\Big(\frac{i}{t}\Big)-\frac{36}{\pi^2}\,e^{2\pi t}+\frac{8640}{\pi}\,t-\frac{18144}{\pi^2}\right)\,e^{-\pi r^2 t}\,dt \Bigg) .\notag\end{align}
The integral converges absolutely for all $r\in\R_{\geq 0}$.
\end{proposition}
\begin{proof}
Suppose that $r>\sqrt{2}$. Then by Proposition~\ref{prop: a(r) double zeroes}
$$a(r)=4i\,\sin(\pi r^2/2)^2\,\int\limits_{0}^{\infty}\phi_0(i/t)\,t^2\,e^{-\pi r^2 t}\,dt. $$
From \eqref{eqn: phi fourier0}--\eqref{eqn: phi0 transform} we obtain
\begin{equation}\label{eqn: phi asymptotic}
\phi_0(i/t)\,t^2=\frac{36}{\pi^2}\,e^{2 \pi t}-\frac{8640}{\pi}\,t+\frac{18144}{\pi^2}+O(t^2\,e^{-2\pi t})\quad\mbox{as}\;t\to\infty.
\end{equation}
For $r>\sqrt{2}$ we have
\begin{equation}
\int\limits_0^\infty \left(\frac{36}{\pi^2}\,e^{2 \pi t}+\frac{8640}{\pi}\,t+\frac{18144}{\pi^2}\right)\,e^{-\pi r^2 t}\,dt
=\frac{36}{\pi^3\,(r^2-2)}-\frac{8640}{\pi^3\,r^4}+\frac{18144}{\pi^3\,r^2}.\end{equation}
Therefore, the identity \eqref{eqn: a another integral} holds for $r>\sqrt{2}$.

On the other hand, from the definition~\eqref{eqn: a(r) definition} we see that $a(r)$ is analytic in some neighborhood of $[0,\infty)$. The asymptotic expansion~\eqref{eqn: phi asymptotic} implies that the right hand side of \eqref{eqn: a another integral} is also analytic in some neighborhood of $[0,\infty)$. Hence, the identity \eqref{eqn: a another integral} holds on the whole interval $[0,\infty)$. This finishes the proof of the proposition.
\end{proof}
From the identity~\eqref{eqn: a another integral} we see that the values $a(r)$ are in $i\R$ for all $r\in\R_{\geq0}$. In particular, we have
\begin{proposition}\label{prop: a values}
We have
\begin{equation}
a(0)=\frac{-i\,8640}{\pi}\qquad
a(\sqrt{2})=0\qquad
a^\prime(\sqrt{2})=\frac{i\,72\sqrt{2}}{\pi}.
\end{equation}
\end{proposition}
\begin{proof}
These identities follow immediately from the previous proposition.
\end{proof}

Now we construct function $b$. To this end we consider the modular form
\begin{equation}\label{eqn: h define}
  h\,:=\,128 \frac{\theta_{00}^4+\theta_{01}^4}{\theta_{10}^8}.
\end{equation}
It is easy to see that $h\in M^!_{-2}(\Gamma_0(2))$. Indeed, first we check that $h|_{-2}\gamma=h$
for all $\gamma\in\Gamma_0(2)$. Since the group $\Gamma_0(2)$ is generated by elements
$\left(\begin{smallmatrix}1&0\\2&1\end{smallmatrix}\right)$ and $\left(\begin{smallmatrix}1&1\\0&1\end{smallmatrix}\right)$
it suffices to check that $h$ is invariant under their action. This follows immediately
from \eqref{eqn: theta transform S}--\eqref{eqn: theta transform T} and \eqref{eqn: h define}. Next we analyze the poles of $h$.
It is known \cite[Chapter~I Lemma~4.1]{Mumford} that $\theta_{10}$ has no zeros in the upper-half plane and hence $h$ has poles only at the cusps.
At the cusp $i\infty$ this modular form has the Fourier expansion
\begin{equation*}
h(z)\,=\,q^{-1} + 16 - 132 q + 640 q^2 - 2550 q^3+O(q^4).
\end{equation*}
Let $I=\left(\begin{smallmatrix}1&0\\0&1\end{smallmatrix}\right)$,
$T=\left(\begin{smallmatrix}1&1\\0&1\end{smallmatrix}\right)$, and
$S=\left(\begin{smallmatrix}0&-1\\1&0\end{smallmatrix}\right)$ be elements of $\Gamma(1)$.
We define the following three functions
\begin{align}
  \psi_I\,:=\,&h-h|_{-2}ST,\label{eqn: psi I define}\\
  \psi_T\,:=\,&\psi_I|_{-2}T ,\label{eqn: psi T define}\\
  \psi_S\,:=\,&\psi_I|_{-2}S. \label{eqn: psi S define}
\end{align}
More explicitly, we have
\begin{align}
\psi_I\,=\,&128\,\frac{\theta_{00}^4+\theta_{01}^4}{\theta_{10}^8}\,+\,128
              \frac{\theta_{01}^4-\theta_{10}^4}{\theta_{00}^8},\label{eqn: psi I explicit}\\
\psi_T\,=\,&128\,\frac{\theta_{00}^4+\theta_{01}^4}{\theta_{10}^8}\,+
              \,128\,\frac{\theta_{00}^4+\theta_{10}^4}{\theta_{01}^8},\label{eqn: psi T explicit}\\
\psi_S\,=\,&-128\,\frac{\theta_{00}^4+\theta_{10}^4}{\theta_{01}^8}-128\,
              \frac{\theta_{10}^4-\theta_{01}^4}{\theta_{00}^8}.\label{eqn: psi S explicit}
\end{align}
The Fourier expansions of these functions are
\begin{align}
  \psi_I(z)\,=\,&q^{-1} + 144 - 5120 q^{1/2} + 70524 q - 626688 q^{3/2} + 4265600 q^2  + O(q^{5/2}) ,\label{eqn: psi fourier I}\\
  \psi_T(z)\,=\,&q^{-1} + 144 + 5120 q^{1/2} + 70524 q + 626688 q^{3/2} + 4265600 q^2  + O(q^{5/2}) ,\label{eqn: psi fourier T}\\
  \psi_S(z)\,=\,&-10240 q^{1/2} - 1253376 q^{3/2} - 48328704 q^{5/2} - 1059078144 q^{7/2}+O(q^{9/2}).\label{eqn: psi fourier S}
\end{align}
For $x\in\R^8$ define
\begin{align}\label{eqn: b(r) definition}
  b(x):= & \int\limits_{-1}^{i}\psi_T(z)\,e^{\pi i \|x\|^2 z}\,dz
    + \int\limits_{1}^{i}\psi_T(z)\,e^{\pi i \|x\|^2 z}\,dz \\
  -& 2\,\int\limits_{0}^{i}\psi_I(z)\,e^{\pi i \|x\|^2 z}\,dz
  - 2\,\int\limits_{i}^{i\infty}\psi_S(z)\,e^{\pi i \|x\|^2 z}\,dz \nonumber.
\end{align}

Now we prove that $b$ satisfies condition \eqref{eqn: b fourier}.
\begin{proposition}\label{prop: b(r) Fourier} The function $b$ defined by \eqref{eqn: b(r) definition} belongs to the Schwartz space and satisfies
  $$\widehat{b}(x)=-b(x). $$
\end{proposition}
\begin{proof}
Here, we repeat the arguments used in the proof of Proposition~\ref{prop: a(r) Fourier}. First we show that $b$ is a Schwartz function. We have
\begin{align*}
 &\int\limits_{-1}^{i}\psi_T(z)\,e^{\pi i r^2 z}\,dz=\int\limits_{0}^{i+1}\psi_I(z)\,e^{\pi i r^2 (z-1)}\,dz=\\
 &\int\limits_{i\infty}^{-1/(i+1)}\psi_I\Big(\frac{-1}{z}\Big)\,e^{\pi i r^2 (-1/z-1)}\,z^{-2}\,dz=\int\limits_{i\infty}^{-1/(i+1)}\psi_S(z)\,z^{-4}\,e^{\pi i r^2 (-1/z-1)}\,dz.\\
\end{align*}
There exists a positive constant $C$ such that
$$|\psi_S(z)|\leq C\,e^{-\pi\,\Im{z}}\quad\mbox{for }\;\Im{z}>\frac12.$$
Thus, as in the proof of Proposition~\ref{prop: a(r) Fourier} we estimate the first summand in the left-hand side of~\eqref{eqn: b(r) definition}
$$\Bigg|\int\limits_{-1}^i \psi_T(z)\,e^{\pi i r^2 z}\,dz \Bigg|\leq C_1\,r\,K_1(2\pi r).$$
We combine this inequality with analogous estimates for the other three summands and obtain
$$|b(r)|\leq C_2\,r\,K_1(2\pi r)+C_3\,\frac{e^{-\pi(r^2+1)}}{r^2+1}.$$
Here $C_1$, $C_2$, and $C_3$ are some positive constants. Similar estimates hold for all derivatives $\frac{d^k}{d^k r} b(r)$.

Now we prove that $b$ is an eigenfunction of the Fourier transform. We use identity~\eqref{eqn: Gaussian Fourier} and interchange contour integration in $z$ and Fourier transform in $x$. Thus we obtain
\begin{align*}
  \mathcal{F}(b)(x)= & \int\limits_{-1}^{i}\psi_T(z)\,z^{-4}\,e^{\pi i \|x\|^2 (\frac{-1}{z})}\,dz
    + \int\limits_{1}^{i}\psi_T(z)\,z^{-4}\,e^{\pi i \|x\|^2 (\frac{-1}{z})}\,dz \\
  -& 2\,\int\limits_{0}^{i}\psi_I(z)\,z^{-4}\,e^{\pi i \|x\|^2 (\frac{-1}{z})}\,dz
  - 2\,\int\limits_{i}^{i\infty}\psi_S(z)\,z^{-4}\,e^{\pi i \|x\|^2 (\frac{-1}{z})}\,dz.
\end{align*}
We make the change of variables $w=\frac{-1}{z}$ and arrive at
\begin{align*}
  \mathcal{F}(b)(x)= & \int\limits_{1}^{i}\psi_T\Big(\frac{-1}{w}\Big)\,w^{2}\,e^{\pi i \|x\|^2 w}\,dw
    + \int\limits_{-1}^{i}\psi_T\Big(\frac{-1}{w}\Big)\,w^{2}\,e^{\pi i \|x\|^2 w}\,dw \\
  -& 2\,\int\limits_{i\infty}^{i}\psi_I\Big(\frac{-1}{w}\Big)\,w^{2}\,e^{\pi i \|x\|^2 w}\,dw
  - 2\,\int\limits_{i}^{0}\psi_S\Big(\frac{-1}{w}\Big)\,w^{2}\,e^{\pi i \|x\|^2 w}\,dw.
\end{align*}
Now we observe that the definitions \eqref{eqn: psi I define}--\eqref{eqn: psi S define}   imply
\begin{align*}\psi_T|_{-2}S=&-\psi_T ,\\
\psi_I|_{-2}S=&\psi_S ,\\
\psi_S|_{-2}S=&\psi_I. \\
\end{align*}
Therefore, we arrive at
\begin{align*}
  \mathcal{F}(b)(x)= & \int\limits_{1}^{i}-\psi_T(z)\,e^{\pi i \|x\|^2 z}\,dz
    + \int\limits_{-1}^{i}-\psi_T(z)\,e^{\pi i \|x\|^2 z}\,dz \\
  +& 2\,\int\limits_{i}^{i\infty}\psi_S(z)\,e^{\pi i \|x\|^2 z}\,dz
  + 2\,\int\limits_{0}^{i}\psi_I(z)\,e^{\pi i \|x\|^2 w}\,dw.
\end{align*}
Now from~\eqref{eqn: b(r) definition} we see that
$$ \mathcal{F}(b)(x)=-b(x). $$
\end{proof}
Now we regard the radial function  $b$ as a function on $\R_{\geq0}$. We check that $b$ has double roots at $\Lambda_8$-points.
\begin{proposition}\label{prop: b(r) double zeroes} For $r>\sqrt{2}$ function $b(r)$ can be expressed as
\begin{equation}\label{eqn: b double zeroes}
  b(r)=-4\sin(\pi r^2/2)^2\,\int\limits_{0}^{i\infty}\psi_I(z)\,e^{\pi i r^2 \,z}\,dz.
\end{equation}
\end{proposition}
\begin{proof}
We denote the right hand side of~\eqref{eqn: b double zeroes} by $c(r)$. First, we check that $c(r)$ is well-defined. We have
\begin{align*}
\psi_I(it)=O(t^{2}\,e^{-\pi/t})\quad\mbox{as}\;t\to 0,\\
 \psi_I(it)=O(e^{2\pi t})\quad\mbox{as}\;t\to\infty.
\end{align*}
Therefore, the integral~\eqref{eqn: b double zeroes} converges for $r>\sqrt{2}$.
Then we rewrite it in the following way:
$$c(r)=\int\limits_{-1}^{i\infty-1}\psi_I(z+1)\,e^{\pi i r^2 \,z}\,dz-2\int\limits_{0}^{i\infty}\psi_I(z)\,e^{\pi i r^2 \,z}\,dz+
\int\limits_{1}^{i\infty+1}\psi_I(z-1)\,e^{\pi i r^2 \,z}\,dz.$$
From the Fourier expansion~\eqref{eqn: psi fourier I} we know that $\psi_I(z)=e^{-2\pi i z}+O(1)$ as $\Im(z)\to\infty$.
By assumption $r^2>2$, hence we can deform the path of integration and write
\begin{align}\label{eqn: inside proof 1}
\int\limits_{-1}^{i\infty-1}\psi_I(z+1)\,e^{\pi i r^2 \,z}\,dz=&
\int\limits_{-1}^{i}\psi_T(z)\,e^{\pi i r^2 \,z}\,dz+\int\limits_{i}^{i\infty}\psi_T(z)\,e^{\pi i r^2 \,z}\,dz,\\
\int\limits_{1}^{i\infty+1}\psi_I(z-1)\,e^{\pi i r^2 \,z}\,dz=&
\int\limits_{-1}^{i}\psi_T(z)\,e^{\pi i r^2 \,z}\,dz+\int\limits_{i}^{i\infty}\psi_T(z)\,e^{\pi i r^2 \,z}\,dz.
\end{align}
We have
\begin{align}\label{eqn: c1}c(r)=&\int\limits_{-1}^{i}\psi_T(z)\,e^{\pi i r^2 \,z}\,dz+\int\limits_{1}^{i}\psi_T(z)\,e^{\pi i r^2 \,z}\,dz
-2\int\limits_{0}^{i}\psi_I(z)\,e^{\pi i r^2 \,z}\,dz\\
&+2\int\limits_{i}^{i\infty}(\psi_T(z)-\psi_I(z))\,e^{\pi i r^2 \,z}\,dz.\nonumber
 \end{align}

 Next, we check that the functions $\psi_I,\psi_T$, and $\psi_S$ satisfy the following identity:
\begin{equation}\label{eqn: c2}\psi_T+\psi_S=\psi_I.\end{equation}
Indeed, from definitions \eqref{eqn: psi I define}-\eqref{eqn: psi S define} we get
\begin{align*}\psi_T+\psi_S=&(h-h|_{-2}ST)|_{-2}T+(h-h|_{-2}ST)|_{-2}S\\
=&h|_{-2}T-h|_{-2}ST^2+h|_{-2}S-h|_{-2}STS.\end{align*}
Note that $ST^2S$ belongs to $\Gamma_0(2)$. Thus, since $h\in M^!_{-2}\Gamma_0(2)$ we get
$$\psi_T+\psi_S=h|_{-2}T-h|_{-2}STS. $$
Now we observe that $T$ and $STS(ST)^{-1}$ are also in $\Gamma_0(2)$. Therefore,
$$\psi_T+\psi_S=h|_{-2}T-h|_{-2}STS=h-h|_{-2}ST=\psi_I.$$

Combining \eqref{eqn: c1} and \eqref{eqn: c2} we find
\begin{align*}c(r)=&\int\limits_{-1}^{i}\psi_T(z)\,e^{\pi i r^2 \,z}\,dz+\int\limits_{1}^{i}\psi_T(z)\,e^{\pi i r^2 \,z}\,dz
-2\int\limits_{0}^{i}\psi_I(z)\,e^{\pi i r^2 \,z}\,dz\\
&-2\int\limits_{i}^{i\infty}\psi_S(z)\,e^{\pi i r^2 \,z}\,dz\\
=&b(r).
 \end{align*}
\end{proof}
At the end of this section we find another integral representation of $b(r)$ for $r\in\R_{\geq0}$ and compute special values of $b$.
\begin{proposition}\label{prop: b another integral}
For $r\geq0$ we have
\begin{equation}\label{eqn: b another integral}b(r)=4i\,\sin(\pi r^2/2)^2\,\left(\frac{144}{\pi\,r^2}+\frac{1}{\pi\,(r^2-2)}+\int\limits_0^\infty\,\left(\psi_I(it)-144-e^{2\pi t}\right)\,e^{-\pi r^2 t}\,dt\right).\end{equation}
The integral converges absolutely for all $r\in\R_{\geq 0}$.
\end{proposition}
\begin{proof}
The proof is analogous to the proof of Proposition~\ref{prop: a another integral}.
First, suppose that $r>\sqrt{2}$. Then by Proposition~\ref{prop: b(r) double zeroes}
$$b(r)=4i\,\sin(\pi r^2/2)^2\,\int\limits_{0}^{\infty}\psi_I(it)\,e^{-\pi r^2 t}\,dt. $$
From \eqref{eqn: psi fourier I} we obtain
\begin{equation}\label{eqn: psi asymptotic}
\psi_I(it)=e^{2\pi t}+144+O(e^{-\pi t})\quad\mbox{as}\;t\to\infty.
\end{equation}
For $r>\sqrt{2}$ we have
\begin{equation}
\int\limits_0^\infty \left(e^{2\pi t}+144\right)\,e^{-\pi r^2 t}\,dt
=\frac{1}{\pi\,(r^2-2)}+\frac{144}{\pi\,r^2}.\end{equation}
Therefore, the identity \eqref{eqn: a another integral} holds for $r>\sqrt{2}$.

On the other hand, from the definition \eqref{eqn: b(r) definition} we see that $b(r)$ is analytic in some neighborhood of $[0,\infty)$. The asymptotic expansion \eqref{eqn: psi asymptotic} implies that the right hand side of \eqref{eqn: b another integral} is also analytic in some neighborhood of $[0,\infty)$. Hence, the identity \eqref{eqn: b another integral} holds on the whole interval $[0,\infty)$. This finishes the proof of the proposition.
\end{proof}
We see from \eqref{eqn: b another integral} that $b(r)\in i\R$ far all $r\in\R_\geq{0}$. Another immediate corollary of this proposition is
\begin{proposition}\label{prop: b values}
We have
\begin{equation}\label{eqn: b values}
b(0)=0\qquad
b(\sqrt{2})=0\qquad
b^\prime(\sqrt{2})=2\sqrt{2}\,\pi\,i.
\end{equation}
\end{proposition}

\section{Proof of Theorem \ref{thm: g}}\label{sec: g}
Finally, we are ready to prove Theorem \ref{thm: g}.
\begin{theorem}\label{thm: g1}
The function
$$g(x):=\frac{\pi\,i}{8640}a(x)+\frac{i}{240\pi}\,b(x)$$
satisfies conditions \eqref{eqn: g1}--\eqref{eqn: g3}. Moreover, the values $g(x)$ and $\widehat{g}(x)$ do not vanish for all vectors $x$ with $\|x\|^2\notin 2\Z_{>0}$.
\end{theorem}
\begin{proof}
First, we prove that \eqref{eqn: g1} holds. By Propositions~\ref{prop: a(r) double zeroes} and \ref{prop: b(r) double zeroes} we know that for $r>\sqrt{2}$
\begin{equation}\label{eqn: g A} g(r)=\frac{\pi}{2160}\,\sin(\pi r^2/2)^2\,\int\limits_0^\infty A(t)\,e^{-\pi r^2 t}\,dt\end{equation}
where $$A(t)=-t^2\phi_0(i/t)-\frac{36}{\pi^2}\,\psi_I(it).$$
Our goal is to show that $A(t)<0\quad\mbox{for}\;t\in(0,\infty).$ The function $A(t)$ is plotted in Figure~\ref{fig: A}.
\begin{figure}[h!]
\caption{Plot of the functions $A(t)$, $A^{(2)}_0(t)=-\frac{368640}{\pi^2}\,t^2\,e^{-\pi /t}$, and $A^{(1)}_\infty(t)=-\frac{72}{\pi^2}\,e^{2\pi t}+\frac{8640}{\pi}t-\frac{23328}{\pi^2}$.\label{fig: A}}
  \centering
\includegraphics[width=300 pt]{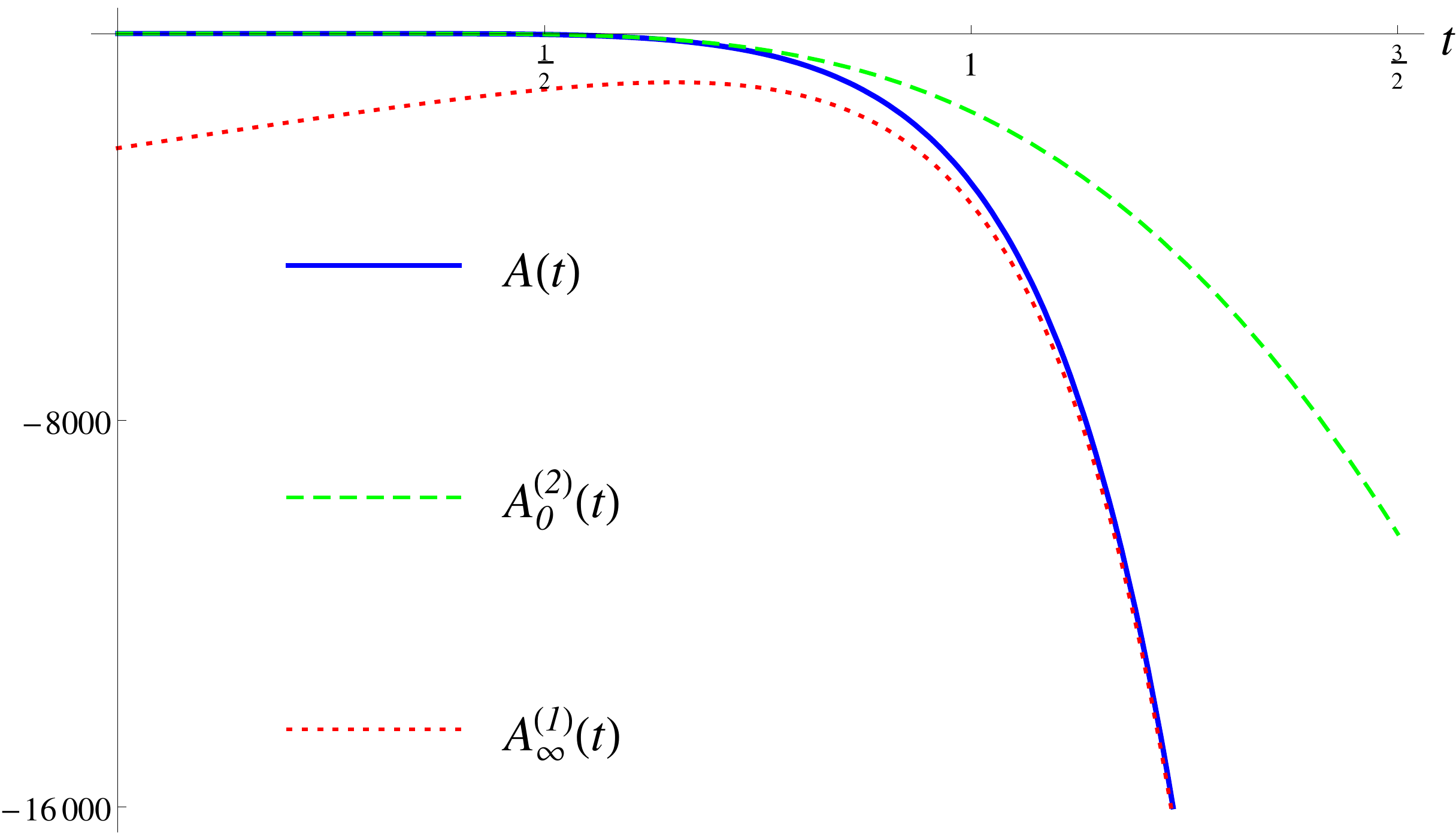}
\end{figure}

\noindent We observe that we can compute the values of $A(t)$ for $t\in(0,\infty)$ with any given precision. Indeed, from identities \eqref{eqn: phi0 transform} and \eqref{eqn: psi S define} we obtain the following two presentations for $A(t)$
\begin{align*}
 A(t)=&-t^2\phi_0(i/t)+\frac{36}{\pi^2}\,t^2\,\psi_S(i/t),\\
 A(t)=&-t^2\phi_0(it)+\frac{12}{\pi}\,t\,\phi_{-2}(it)-\frac{36}{\pi^2}\,\phi_{-4}(it)-\frac{36}{\pi^2}\,\psi_I(it).
\end{align*}
For an integer $n\geq0$ let $A_0^{(n)}$ and  $A_{\infty}^{(n)}$ be the functions such that
\begin{align}
 A(t)=&A_0^{(n)}(t)+O(t^2\,e^{-\pi n /t})\quad\mbox{as}\;t\to0,\label{eqn: A asymptotic expansion 0}\\
 A(t)=&A_\infty^{(n)}(t)+O(t^2\,e^{-\pi n t})\quad\mbox{as}\;t\to\infty.\label{eqn: A asymptotic expansion infty}
\end{align}
For each $n\geq 0$ we can compute these functions from the Fourier expansions \eqref{eqn: phi fourier0}--\eqref{eqn: phi fourier4}, \eqref{eqn: psi fourier I}, and \eqref{eqn: psi fourier S}.
 For example, from \eqref{eqn: phi fourier4}--\eqref{eqn: phi fourier0} and \eqref{eqn: psi fourier I} we compute
%$$A_\infty^{(0)}(t)=-\frac{72}{\pi^2}\,e^{2\pi t}$$ and
\begin{align*}A_\infty^{(6)}(t)=&\scriptstyle-\tfrac{72}{\pi ^2}\, e^{2 \pi  t}-\tfrac{23328}{\pi ^2}+\tfrac{184320}{\pi ^2}\, e^{-\pi  t}-\tfrac{5194368}{\pi ^2}\, e^{-2 \pi  t}+\tfrac{22560768}{\pi ^2}\, e^{-3 \pi  t}-\tfrac{250583040}{\pi
   ^2}\, e^{-4 \pi  t}+\tfrac{869916672 }{\pi ^2}\,e^{-5 \pi  t}\\&\scriptstyle+t(\tfrac{8640}{\pi }+\tfrac{2436480}{\pi }\, e^{-2 \pi  t}+\tfrac{113011200 }{\pi }\,e^{-4 \pi  t})-t^2(518400\,e^{-2 \pi  t}+31104000\,e^{-4 \pi  t}).
\end{align*}
From \eqref{eqn: phi fourier4}--\eqref{eqn: phi fourier0} and \eqref{eqn: psi fourier S} we compute
%$$A_0^{(2)}(t)=-\frac{368640}{\pi^2}\,t^2\,e^{-\pi /t}$$and
$$A_0^{(6)}(t)=t^2(-\tfrac{368640}{\pi ^2}\, e^{-\pi/t}-518400\, e^{-2\pi/t}-\tfrac{45121536}{\pi ^2}\, e^{-3\pi/t}-31104000\,e^{-4\pi/t}-\tfrac{1739833344}{\pi ^2}\, e^{-5\pi/t}).$$
Moreover, from the convergent asymptotic expansion for the Fourier coefficients of a weakly holomorphic modular form \cite[Proposition 1.12]{Bruinier} we find that the $n$-th Fourier coefficient $c_{\psi_I}(n)$ of $\psi_I$ satisfies
\begin{equation}\label{eqn: Fourier estimate 1}|c_{\psi_I}(n)|\leq e^{4\pi\sqrt{n}}\qquad n\in\tfrac 12 \Z_{>0}.\end{equation} Similar inequalities hold for the Fourier coefficients of $\psi_S$, $\phi_0$, $\phi_{-2}$, and $\phi_{-4}$:
\begin{align}\label{eqn: Fourier estimate 2}
&|c_{\psi_S}(n)|\leq 2e^{4\pi\sqrt{n}}\qquad n\in\tfrac 12 \Z_{>0} ,\\
&|c_{\phi_0}(n)|\leq 2e^{4\pi\sqrt{n}}\qquad n\in \Z_{>0} ,\\
&|c_{\phi_{-2}}(n)|\leq e^{4\pi\sqrt{n}}\qquad n\in  \Z_{>0},\\
&|c_{\phi_{-4}}(n)|\leq e^{4\pi\sqrt{n}}\qquad n\in \Z_{>0}. \label{eqn: Fourier estimate 5}
 \end{align}
Therefore, we can estimate the error terms in the asymptotic expansions \eqref{eqn: A asymptotic expansion 0} and \eqref{eqn: A asymptotic expansion infty} of $A(t)$
\begin{align*}
\left|A(t)-A_0^{(m)}(t)\right|\leq& (t^2+\frac{36}{\pi^2})\,\sum_{n=m}^\infty 2e^{2\sqrt{2}\pi\sqrt{n}}\,e^{-\pi n/t},\\
\left|A(t)-A_\infty^{(m)}(t)\right|\leq& (t^2+\frac{12}{\pi}\,t+\frac{36}{\pi^2})\,\sum_{n=m}^\infty 2e^{2\sqrt{2}\pi\sqrt{n}}\,e^{-\pi nt}.
\end{align*}
 For an integer $m\geq0$ we set
\begin{align*}
R^{(m)}_0:=&(t^2+\frac{36}{\pi^2})\,\sum_{n=m}^\infty 2e^{2\sqrt{2}\pi\sqrt{n}}\,e^{-\pi n/t},\\
R^{(m)}_\infty:=&(t^2+\frac{12}{\pi}\,t+\frac{36}{\pi^2})\,\sum_{n=m}^\infty 2e^{2\sqrt{2}\pi\sqrt{n}}\,e^{-\pi nt}.
\end{align*}
Using interval arithmetic we check that
\begin{align*}
&\left|R_0^{(6)}(t)\right|\leq\left|A_0^{(6)}(t)\right|\quad\mbox{ for }\;t\in(0,1],\\
&\left|R_\infty^{(6)}(t)\right|\leq\left|A_\infty^{(6)}(t)\right|\quad\mbox{ for }\;t\in[1,\infty),\\
&A_0^{(6)}(t)<0\quad\mbox{ for }\;t\in(0,1],\\
&A_\infty^{(6)}(t)<0\quad\mbox{ for }\;t\in[1,\infty).
\end{align*}
Thus, we see that $A(t)<0$ for $t\in (0,\infty)$. Then identity \eqref{eqn: g A} implies \eqref{eqn: g1}.

Next, we prove \eqref{eqn: g2}. By Propositions~\ref{prop: a another integral} and~\ref{prop: b another integral} we know that for $r>0$
\begin{equation}\label{eqn: g B} \widehat{g}(r)=\frac{\pi}{2160}\,\sin(\pi r^2/2)^2\,\int\limits_0^\infty B(t)\,e^{-\pi r^2 t}\,dt\end{equation}
where $$B(t)=-t^2\phi_0(i/t)+\frac{36}{\pi^2}\,\psi_I(it).$$
This function can also be written as
\begin{align*}
 B(t)=&-t^2\phi_0(i/t)-\frac{36}{\pi^2}\,t^2\,\psi_S(i/t),\\
 B(t)=&-t^2\phi_0(it)+\frac{12}{\pi}\,t\,\phi_{-2}(it)-\frac{36}{\pi^2}\,\phi_{-4}(it)+\frac{36}{\pi^2}\,\psi_I(it).
\end{align*}
Our aim is to prove that $B(t)>0$ for $t\in(0,\infty)$. A plot of $B(t)$ is given in Figure~\ref{fig:B}.
\begin{figure}[h!]
\caption{Plot of the functions $B(t)$, $B^{(2)}_0(t)=\frac{368640}{\pi^2}\,t^2\,e^{-\pi /t}$, and $B^{(1)}_\infty(t)=\frac{8640}{\pi}t-\frac{23328}{\pi^2}$.\label{fig:B}}
  \centering
\includegraphics[width=300 pt]{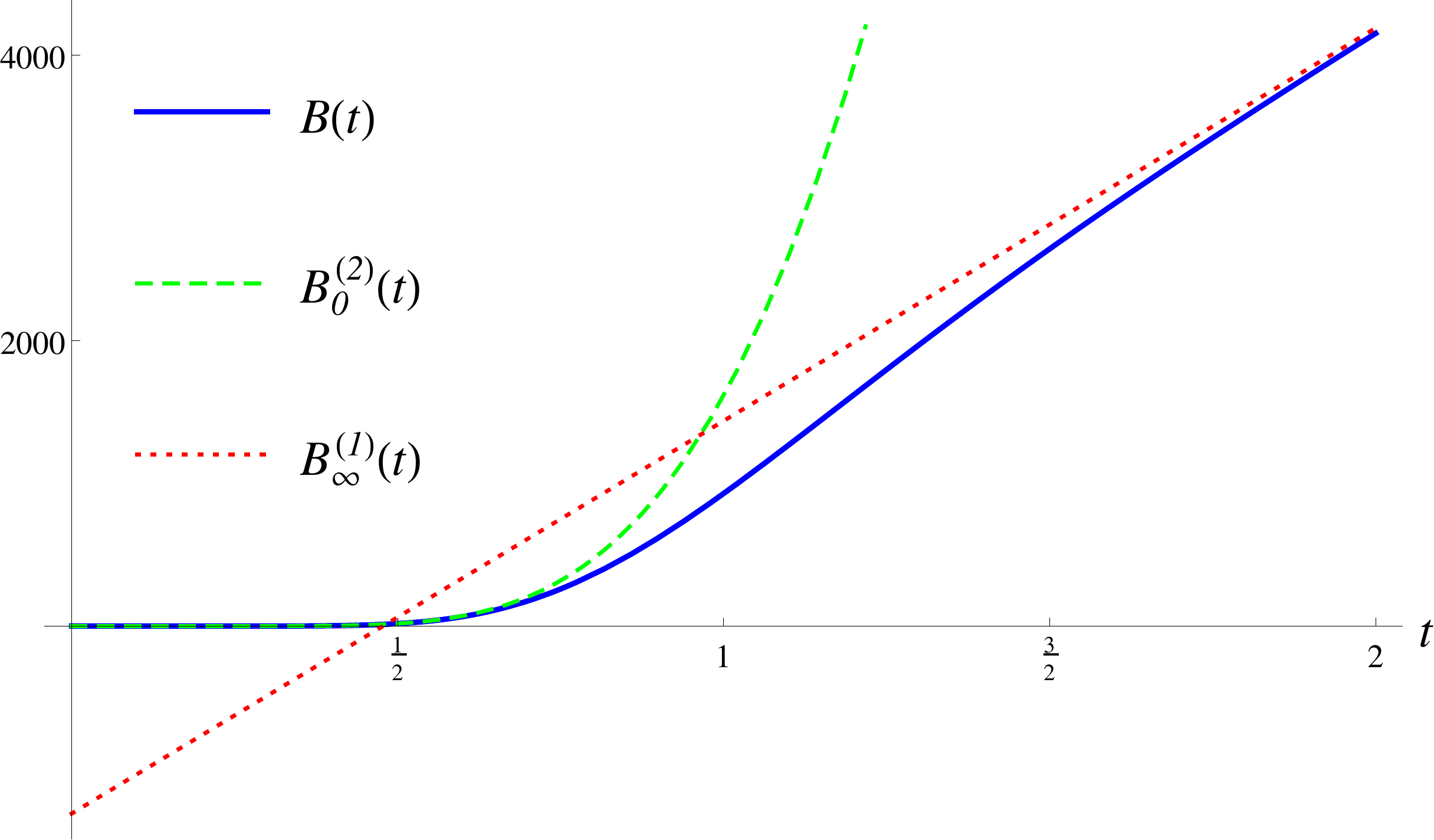}
\end{figure}

\noindent For $n\geq 0$ let $B_0^{(n)}$ and  $B_{\infty}^{(n)}$ be the functions  such that
\begin{align*}
 B(t)=&B_0^{(n)}(t)+O(t^2\,e^{-\pi n /t})\quad\mbox{as}\;t\to0,\\
 B(t)=&B_\infty^{(n)}(t)+O(t^2\,e^{-\pi n t})\quad\mbox{as}\;t\to\infty.
\end{align*}
We find
\begin{align*}B_\infty^{(6)}(t)=&-\tfrac{12960}{\pi ^2}-\tfrac{184320}{\pi ^2}\,\scriptstyle e^{-\pi  t}\displaystyle-\tfrac{116640}{\pi ^2} \,\scriptstyle e^{-2\pi  t}\displaystyle-\tfrac{22560768}{\pi ^2}\,\scriptstyle e^{-3\pi  t}\displaystyle+\tfrac{56540160}{\pi ^2}\,\scriptstyle e^{-4\pi  t}\displaystyle-\tfrac{869916672 }{\pi ^2} \,\scriptstyle e^{-5\pi  t}\displaystyle\\&+t(\tfrac{8640 }{\pi }+\tfrac{2436480}{\pi }\,\scriptstyle e^{-2\pi  t}\displaystyle+\tfrac{113011200 }{\pi }\,\scriptstyle e^{-4\pi  t}\textstyle)-t^2(\scriptstyle518400\,\scriptstyle e^{-2\pi  t}\displaystyle+\scriptstyle31104000\,\scriptstyle e^{-4\pi  t}\displaystyle)
\end{align*}
and
$$B_0^{(6)}(t)= t^2(\tfrac{368640}{\pi ^2}\, e^{-\pi/t}-518400\, e^{-2 \pi /t}+\tfrac{45121536 }{\pi ^2}\,e^{-3 \pi/t}-31104000\, e^{-4 \pi/t}+\tfrac{1739833344 }{\pi ^2}\,e^{-5 \pi/t}) .$$
The estimates \eqref{eqn: Fourier estimate 1}--\eqref{eqn: Fourier estimate 5} imply that $$\left|B(t)-B_0^{(6)}(t)\right|\leq R_0^{(6)}(t)\quad\mbox{for}\;t\in(0,1]$$
and
$$\left|B(t)-B_\infty^{(6)}(t)\right|\leq R_\infty^{(6)}(t)\quad\mbox{for}\;t\in[1,\infty).$$
Using interval arithmetic we verify that
\begin{align*}
&\left|R_0^{(6)}(t)\right|\leq\left|B_0^{(6)}(t)\right|\quad\mbox{ for }\;t\in(0,1],\\
&\left|R_\infty^{(6)}(t)\right|\leq\left|B_\infty^{(6)}(t)\right|\quad\mbox{ for }\;t\in[1,\infty),\\
&B_0^{(6)}(t)>0\quad\mbox{ for }\;t\in(0,1],\\
&B_\infty^{(6)}(t)>0\quad\mbox{ for }\;t\in[1,\infty).
\end{align*}
Now identity \eqref{eqn: g B} implies \eqref{eqn: g2}.

Finally, the property \eqref{eqn: g3} readily follows from Proposition~\ref{prop: a values} and Proposition~\ref{prop: b values}.
This finishes the proof of Theorems~\ref{thm: g1} and~\ref{thm: g}.
 \end{proof}

 \section*{Acknowledgments}
I  thank Andriy Bondarenko for sharing his ideas, for fruitful discussions, and for his support. Also I am grateful to Danilo Radchenko for his valuable ideas and his help with numerical computations. I am most grateful to J. Kramer, J. M. Sullivan, G. M. Ziegler , and anonymous referees for their valuable comments and suggestions on the manuscript.

{\footnotesize
\noindent
Berlin Mathematical School\\
Str. des 17. Juni 136\\
10623 Berlin\\
and\\
Humboldt University of Berlin\\
Rudower Chaussee 25\\
12489 Berlin\\
\textit{ Email address: viazovska@gmail.com}}

\end{document}